\documentclass[smallextended]{svjour3}
\usepackage{amsmath,amsfonts,amssymb}
\usepackage{graphicx,color}
\usepackage[all]{xy}
\usepackage{cite}
\usepackage{amsfonts}
\usepackage[weather,alpine,misc,geometry]{ifsym}
\newcommand{\norm}[1]{\left\Vert#1\right\Vert}
\newcommand{\TO}[1]{\stackrel{#1}{\to}}
\newcommand{\R}{\mathbb R}
\newcommand{\N}{\mathbb N}
\newcommand{\B}{\mathbb B}

\def\RHS{right-hand side}

\def\smartqedtr{\def\qedtr{\ifmmode\triangle\else{\unskip\nobreak\hfil
\penalty50\hskip1em\null\nobreak\hfil$\triangle$
\parfillskip=0pt\finalhyphendemerits=0\endgraf}\fi}}
\smartqedtr  
\smartqed
\lccode`\-=`\-



\begin{document}
\title
{About uniform regularity of collections of sets}
\author
{Alexander Y. Kruger \and Nguyen H. Thao}

\dedication{Dedicated to Asen Dontchev on the occasion of his 65th birthday and  Vladimir  Veliov on the occasion of his 60th birthday}

\institute{
Alexander Y. Kruger (\Letter\,)\at
Centre for Informatics and Applied Optimization,
School of Science, Information Technology and Engineering,
University of Ballarat, Ballarat, Victoria, Australia\\
\email{a.kruger@ballarat.edu.au}
\and
Nguyen H. Thao \at
Centre for Informatics and Applied Optimization,
School of Science, Information Technology and Engineering,
University of Ballarat, Ballarat, Victoria, Australia\\
\email{hieuthaonguyen@students.ballarat.edu.au, nhthao@ctu.edu.vn}
}

\date{Received: date / Accepted: date}
\journalname{}
\maketitle

\begin{abstract}
We further investigate the uniform regularity property of collections of sets via primal and dual characterizing constants. These constants play an important role in determining convergence rates of projection algorithms for solving feasibility problems.
\end{abstract}
\thispagestyle{empty}
\section{Introduction}
Regularity properties of collections of sets play an important role in several areas of variational analysis and optimization like coderivative-subdifferential calculus, constraint qualifications, stability of solutions, and convergence of numerical algorithms.

Various regularity properties of collections of sets have proved to be useful: \emph{(bounded) linear regularity} \cite{BauBor93,BauBor96,BakDeuLi05,BauBorLi99,ZheNg08, LiNgPon07,NgZan07,ZheWeiYao10,BauBorTse00,BurDeng05}, \emph{metric inequality} \cite{Iof89,Iof00_,NgaiThe01}, \emph{(strong) conical hull intersection property} \cite{ChuDeuWar90,ChuDeuWar92,DeuLiWar99,BauBorLi99, BakDeuLi05,LiNgPon07,BauBorTse00}, \emph{Jameson's property $(G)$} \cite{BauBorLi99,Kus87}.
We refer the readers to \cite{BauBorLi99,Kru06.1,BakDeuLi05} for the relationships between these properties and the overview of the areas of their applications in analysis and optimization.

The \emph{uniform regularity} property introduced recently in \cite{Kru05.1} and further developed in \cite{Kru06.1,Kru09.1,KruLop12.1} is stronger than local linear regularity even in the convex case.
It corresponds to the \emph{metric regularity} property of set-valued mappings and is closely related to the \emph{(extended) extremal principle}.
The most recent development is the application of this property in convergence analysis of projection algorithms by Lewis et al. \cite{LewLukMal09}, Attouch et al. \cite{AttBolRedSou10}, Luke \cite{Luk12,Luk13}, and Hesse and Luke \cite{HesLuk}.

Uniform regularity of a collection of sets in a normed linear space is characterized quantitatively in \cite{Kru05.1,Kru06.1,Kru09.1,KruLop12.1} by certain nonnegative constants defined in terms of elements of the primal or dual spaces.
In the setting of a finite dimensional Euclidean space, Lewis et al. \cite{LewLukMal09} introduced another nonnegative constant characterizing the uniform regularity of a collection of two sets and used it when formulating convergence rates of averaged and alternating projections.

In the current note, we consider a (not necessarily nonnegative) modification of the constant from \cite{LewLukMal09} in the setting of an arbitrary Hilbert space and establish its relationship with the dual space constant from \cite{Kru05.1,Kru06.1,Kru09.1,KruLop12.1}.
The latter constant admits a simplified equivalent representation in Hilbert spaces.
As an application, we employ these constants to establish convergence results of projection algorithms.

The structure of the paper is as follows.
In Section \ref{S20}, we recall the uniform regularity property of a finite collection of sets in a normed linear space, its main characterizations and connections with some other properties.
In Section~\ref{S2}, we consider the case of a collection of two sets in a Hilbert space and establish the relationship between the dual space constants from \cite{Kru05.1,Kru06.1,Kru09.1,KruLop12.1} and \cite{LewLukMal09}.
The final Section \ref{S3} is dedicated to the convergence estimates of projection algorithms.

Our basic notation is standard, cf. \cite{Mor06.1,RocWet98}.
For a normed linear space $X$, its topological dual is denoted $X^*$ while $\langle\cdot,\cdot\rangle$ denotes the bilinear form defining the pairing between the two spaces.
The closed unit ball and the unit sphere in a normed space are denoted $\B$ and $\mathbb{S}$, respectively.
$B_\delta(x)$ stands for the closed ball with radius $\delta$ and center $x$.

\section{Uniform regularity of a collection of sets}\label{S20}
In this section, we recall the \emph{uniform regularity} property of a finite collection $\bold{\Omega}:=\{\Omega_1,\Omega_2,\ldots,\Omega_m\}$ ($m>1$) of sets in a normed linear space $X$ near a given point $\bar{x}\in \cap_{i=1}^m\Omega_i$.
The property was introduced in \cite{Kru05.1} (under a different name) and further developed in \cite{Kru06.1,Kru09.1,KruLop12.1}.
\begin{definition}\label{defUR}
$\bold{\Omega}$ is \emph{uniformly regular} at $\bar{x}$ if there exist numbers $\delta,\alpha >0$ such that
$$
\bigcap_{i=1}^m(\Omega_i-\omega_i-a_i)\  \bigcap\  (\rho\mathbb{B}) \neq \emptyset
$$
for any $\rho \in (0;\delta], \omega_i \in \Omega_i \cap B_{\delta}(\bar{x})$ and $a_i \in (\alpha\rho)\mathbb{B}$, $i=1,2,\ldots,m$.
\end{definition}

Uniform regularity of a collection of sets can be equivalently characterized in terms of certain nonnegative constants:
$$
\theta_{\rho}[\bold{\Omega}](\bar{x}):= \sup\left\{r \geq 0\left|\; \bigcap_{i=1}^m (\Omega_i - a_i) \bigcap B_{\rho}(\bar{x}) \neq \emptyset, \max_{1\le i\le m}\|a_i\|\le r\right.\right\},\;\rho \in (0;\infty],
$$
$$
\hat{\theta}[\bold{\Omega}](\bar{x}):= \liminf_{\rho \downarrow 0,\;\omega_i \TO{\Omega_i} \bar{x}\; (1\le i\le m)} \dfrac{\theta_{\rho}[\Omega_1-\omega_1,\Omega_2-\omega_2,\ldots,\Omega_m-\omega_m](0)}{\rho}.
$$
Here $\omega_i \TO{\Omega_i} \bar{x}$ means that $\omega_i\to\bar{x}$ with $\omega_i\in\Omega_i$.

These constants characterize the mutual arrangement of sets $\Omega_i$ $(1\le i\le m)$ in the primal space and are convenient for defining their extremality, stationarity and regularity properties.

The next proposition follows directly from the definitions.
\begin{proposition}
$\bold{\Omega}$ is uniformly regular at $\bar{x}$ if and only if $\hat{\theta}[\bold{\Omega}](\bar{x})>0$.
\end{proposition}

When constant $\hat{\theta}[\bold{\Omega}](\bar{x})$ is positive, it provides a quantitative characterization of the uniform regularity property.
It coincides with the supremum of all $\alpha$ in Definition~\ref{defUR}.

The case $\hat{\theta}[\bold{\Omega}](\bar{x})=0$, i.e., the absence of the uniform regularity, corresponds to \emph{approximate stationarity} \cite{Kru03.1,Kru04.1,Kru05.1,Kru06.1,Kru09.1} of $\bold{\Omega}$ at $\bar x$, the latter property being a relaxation of the  \emph{extremality} property introduced and investigated in \cite{KruMor80.3}.
We refer the reader to \cite[Section~3]{KruLop12.1} for a modern summary of extremality, stationarity, and regularity conditions for finite collections of sets.

Another nonnegative primal space constant (being a slight modification of the corresponding one introduced in \cite{Kru05.1}) can be used for characterizing the uniform regularity:
$$
\hat{\vartheta}[\bold{\Omega}](\bar{x}):= \liminf_{\substack{x\to\bar{x},\;x_i \to 0\;(1\le i\le m)\\x\notin\bigcap\limits_{i=1}^m(\Omega_i-x_i)}} \frac{\max\limits_{1\leq i\leq m}d(x+x_i,\Omega_i)} {d\left(x,\bigcap\limits_{i=1}^m(\Omega_i-x_i)\right)}.
$$

The next proposition corresponds to \cite[Theorem~1]{Kru05.1}.

\begin{proposition}\label{dualconstant}
$\hat{\theta}[\bold{\Omega}](\bar{x}) = \hat{\vartheta}[\bold{\Omega}](\bar{x})$.

As a consequence, $\bold{\Omega}$ is uniformly regular at $\bar{x}$ if and only if $\hat{\vartheta}[\bold{\Omega}](\bar{x})>0$.
\end{proposition}

It was shown in \cite{Kru05.1,Kru06.1,Kru09.1} that the uniform regularity of a collection of sets can be interpreted as the direct analogue of the fundamental in variational analysis \emph{metric regularity} property of set-valued mappings.

Regularity properties can also be characterized in terms of elements of the dual space using appropriate concepts of \emph{normal} elements.
Given a subset $\Omega$ of $X$, a point $\bar x$ in $\Omega$, and a number $\delta \geq 0$, the sets (cf. \cite{Kru03.1,Mor06.1})
\begin{gather*}
N_{\Omega}(\bar x) := \left\{x^* \in X^* \mid \limsup_{x\TO{\Omega}\bar x} \frac {\langle x^*,x-\bar x \rangle}{\|x-\bar x\|} \leq 0 \right\},
\\
\widehat N_\Omega(\bar x,\delta):= \bigcup_{x\in\Omega\cap B_{\delta}(\bar x)} N_\Omega(x),
\\
\overline{N}_{\Omega}(\bar x) :=  \limsup_{x \TO{\Omega} \bar x}N_{\Omega}(x)= \bigcap_{\delta>0}{\rm cl}^*\widehat N_{\Omega}(\bar x,\delta)
\end{gather*}
denote the \emph{Fr\'echet normal cone}, the \emph{strict $\delta$-normal cone}, and the \emph{limiting normal cone} to $\Omega$ at $\bar x$, respectively.
The denotation $u\TO{\Omega}x$ in the above formulas means that $u\to{x}$ with $u\in\Omega$ while ${\rm cl}^*$ denotes the sequential weak$^*$ closure in $X^*$.

In the Asplund space setting, the uniform regularity of a collection of sets can be characterized using the next dual space constant:
\begin{equation}\label{constant1'}
\hat{\eta}[\bold{\Omega}](\bar{x}):= \lim_{\delta\downarrow 0}\inf\left\{\norm{\sum_{i=1}^m x_i^*} \mid x_i^* \in \widehat N_{\Omega_i}(\bar x,\delta), \sum_{i=1}^m\norm{x_i^*} =1\right\},
\end{equation}
where it is assumed that the infimum over the empty set equals $1$;
this corresponds to all cones $\widehat N_{\Omega_i}(\bar x,\delta)$ $(1\le i\le m)$ being trivial for some $\delta>0$ ($\bar x$ can be an interior point of $\cap_{i=1}^m\Omega_i$.)

The next theorem corresponds to \cite[Theorem~4 (v)--(vi)]{Kru09.1}.

\begin{theorem}\label{ol0}
\begin{enumerate}
\item
$\hat\theta[\bold{\Omega}](\bar{x})\le
\hat\eta[\bold{\Omega}](\bar{x})$.
\item
Suppose $X$ is Asplund and the sets $\Omega_i$ $(1\le i\le{m})$ are closed.
Then
$\hat\theta[\bold{\Omega}](\bar{x})=
\hat\eta[\bold{\Omega}](\bar{x})$.

As a consequence, $\bold{\Omega}$ is uniformly regular at $\bar{x}$ if and only if $\hat{\eta}[\bold{\Omega}](\bar{x})>0$, i.e.,
there exist $\alpha>0$ and $\delta>0$ such that
\begin{equation}\label{NUR}
\left\|\sum\limits_{i=1}^m{x}_i^*\right\|\ge
\alpha\sum\limits_{i=1}^m\|x_i^*\|
\end{equation}
for all $x_i\in\Omega_i\cap{B}_\delta(\bar{x})$ and
$x_i^*\in{N}_{\Omega_i}(x_i)$ ($1\le i\le{m}$).
\end{enumerate}
\end{theorem}

The dual characterization of the uniform regularity in Theorem~\ref{ol0} (ii) is sometimes referred to as \emph{(Fr\'echet) normal uniform regularity}, cf. \cite{Kru09.1,KruLop12.1}.
Constant $\hat{\eta}[\bold{\Omega}](\bar{x})$ coincides with the supremum of all $\alpha$ in the definition of this property.

Part (i) of Theorem~\ref{ol0} was proved in \cite{Kru04.1}, while part (ii) was established in \cite{Kru09.1}. A slightly weaker estimate can be found in \cite{Kru04.1,Kru06.1}.

\begin{remark}
In finite dimensions, constant \eqref{constant1'} coincides with the corresponding one defined in terms of limiting normals:
\begin{align}\notag
\bar\eta[\bold{\Omega}](\bar x):=& \min\left\{\norm{\sum_{i=1}^m x_i^*} \mid x_i^* \in \overline{N}_{\Omega_i}(\bar x), \sum_{i=1}^m\norm{x_i^*} =1\right\}
\end{align}
(with the similar natural convention about the minimum over the empty set.) The dual uniform regularity criterion in Theorem~\ref{ol0} (ii) takes the following ``exact'' (``at the point'') form:
\begin{center}
\emph{there exists $\alpha>0$ such that \eqref{NUR} holds true for all
$x_i^*\in\overline{N}_{\Omega_i}(\bar x)$ ($1\le i\le{m}$)},
\end{center}
or equivalently,
\begin{equation*}
\left.
\begin{array}{ll}
x_i^*\in\overline{N}_{\Omega_i}(\bar x)\; (1\le i\le{m})\\
x_1^*+x_2^*+\ldots+x_n^*=0
\end{array}
\right\}
\quad\Longrightarrow\quad x_1^*=x_2^*=\ldots=x_n^*=0.
\end{equation*}
This is a well known qualification condition, cf. \cite[Corollary 3.37]{Mor06.1}.
\end{remark}

Apart from the formulated in Theorem~\ref{ol0} (ii) necessary and sufficient characterization of the uniform regularity, equality $\hat\theta[\bold{\Omega}](\bar{x})=
\hat\eta[\bold{\Omega}](\bar{x})$ implies also an equivalent characterization of approximate stationarity.
\begin{corollary}[Extended extremal principle \cite{Kru03.1,Kru04.1}]\label{EEP}
Suppose $X$ is Asplund and the sets $\Omega_i$ $(1\le i\le{m})$ are closed.
$\bold{\Omega}$ is approximately stationary at $\bar{x}$ if and only if $\hat{\eta}[\bold{\Omega}](\bar{x})=0$, i.e.,
for any $\varepsilon>0$ there exist
$x_i\in\Omega_i\cap{B}_\varepsilon(\bar{x})$ and
$x_i^*\in{N}_{\Omega_i}(x_i)$ ($1\le i\le{m}$) such that
\begin{equation*}
\left\|\sum\limits_{i=1}^m{x}_i^*\right\|<\varepsilon
\sum\limits_{i=1}^m\|x_i^*\|.
\end{equation*}
\end{corollary}

This result extends the \emph{extremal principle} \cite{KruMor80.3,MorSha96.1} and can be considered as a generalization of the convex \emph{separation theorem} to collections of nonconvex sets.
Some earlier formulations of Corollary~\ref{EEP} can be found in \cite{Kru98.1,Kru00.1,Kru02.1}.

\begin{remark}
Corollary~\ref{EEP} provides also an equivalent characterization of Asplund spaces, cf. \cite[Theorem~5]{Kru09.1}.
Theorem~\ref{ol0} (ii) can be extended from Asplund to arbitrary Banach spaces if Fr\'echet normal cones are replaced by some other kind of normal cones satisfying certain natural properties, e.g., Clarke normal cones, cf. \cite{KruLop12.1}.
\end{remark}

\begin{remark}
Theorem~\ref{ol0} can be extended to infinite collections of sets.
This allows us to treat infinite and semi-infinite optimization problems, cf. \cite{KruLop12.1,KruLop12.2}.
\end{remark}

Verifying the uniform regularity (and several other properties) of a finite collection of sets can always be reduced to that of two sets in the product space.

\begin{proposition}[\cite{Kru05.1}, Proposition~4]\label{mTO2}
$\bold{\Omega}$ is uniformly regular at $\bar x$ if and only if the collection of two sets
\begin{equation}\label{product}
\Omega:=\Omega_1\times\Omega_2\times\ldots\times\Omega_m \quad\mbox{and}\quad L:=\{(x,x,\ldots,x)\mid x\in X\}
\end{equation}
in $X^m$ (with any norm compatible with that in $X$) is uniformly regular at the point $(\bar x,\bar x,\ldots,\bar x)$.
\end{proposition}

Note the following simple representations of the Fr\'echet normal cones to the sets in \eqref{product}.

\begin{proposition}\label{P4}
\begin{enumerate}
\item
Suppose $x_i\in\Omega_i$ ($1\le i\le m$). Then
$$
N_{\Omega}(z)=\prod_{i=1}^mN_{\Omega_i}(x_i),
$$
where $z=(x_1,x_2,\ldots,x_m)$.
\item
Suppose $x\in X$. Then
$$
N_{L}(z)=L^\perp =\left\{z^*=(x_1^*,\ldots,x_m^*)\in (X^*)^m\mid \sum_{i=1}^m x_i^*=0\right\},
$$
where $z=Ax:= (x,x,\ldots,x)$.
\end{enumerate}
\end{proposition}

\begin{proof}
The first assertion follows directly from the definition while proving the second one is a simple exercise on application of standard tools of convex analysis.
\qed\end{proof}

\section{Uniform regularity in a Hilbert space}\label{S2}

In this section, we limit ourselves to the case when $X$ is a Hilbert space.
For the collection of sets $\bold{\Omega}=\{\Omega_1,\Omega_2,\ldots,\Omega_m\}$ ($m>1$), denote
\begin{equation}\label{ets}
\hat c[\bold{\Omega}](\bar x) :=1-2(\hat\eta[\bold{\Omega}](\bar x))^2,
\end{equation}
where $\hat\eta[\bold{\Omega}](\bar x)$ is the dual space regularity constant defined by \eqref{constant1'}.
By Theorem~\ref{ol0}~(ii), the uniform regularity of $\bold{\Omega}$ at $\bar x$ is equivalent to the inequality $\hat c[\bold{\Omega}](\bar x)<1$.
Note that constant \eqref{ets} can be negative: $\hat c[\bold{\Omega}](\bar x)\ge-1$.

\begin{lemma}\label{l1}
Suppose $\bold{\Omega}$ is uniformly regular at $\bar x$. Then, for any $c'>\hat c[\bold{\Omega}](\bar x)$, there is $\delta>0$ such that, for any $i,j\in\{1,2,\ldots,m\}$, $i\ne j$, and any $u\in\widehat N_{\Omega_i}(\bar x,\delta)\cap \mathbb{S},v\in \widehat N_{\Omega_j}(\bar x,\delta)\cap \mathbb{S}$, it holds:
\begin{equation}\label{lta}
-\langle u,v\rangle<c'.
\end{equation}
\end{lemma}
\begin{proof}
By definition \eqref{constant1'}, for any $c'>\hat c[\bold{\Omega}](\bar x))$, there is $\delta>0$ such that
$$
2\norm{\sum_{k=1}^m x_k^*}^2>1-c'\mbox{ for all }x_k^* \in \widehat N_{\Omega_k}(\bar x,\delta)\mbox{ with } \sum_{k=1}^m\norm{x_k^*} =1.
$$
Choose any $i,j\in\{1,2,\ldots,m\}$, $i\ne j$, and any $u\in\widehat N_{\Omega_i}(\bar x,\delta)\cap \mathbb{S},v\in \widehat N_{\Omega_j}(\bar x,\delta)\cap \mathbb{S}$.
Set $x_i^*=u/2$, $x_j^*=v/2$, and $x_k^*=0$ for $k\in\{1,2,\ldots,m\}\setminus\{i,j\}$.
Then $x_k^* \in \widehat N_{\Omega_k}(\bar x,\delta)$ ($k\in\{1,2,\ldots,m\}$) and $\sum_{k=1}^m\norm{x_k^*} =1$.
It follows that
$$
\|u+v\|^2=4\norm{\sum_{k=1}^m x_k^*}^2>2(1-c'),
$$
or equivalently
$$
2+2\langle u,v\rangle>2(1-c').
$$
In its turn, the last inequality is equivalent to \eqref{lta}.
\qed\end{proof}

In the rest of the section, we assume that $m=2$, i.e., $\bold{\Omega}=\{\Omega_1,\Omega_2\}$.
Definition \eqref{constant1'} of the constant characterizing the uniform regularity of a collection of sets can be simplified.

\begin{proposition}\label{n=2}
The following representation holds true:
\begin{equation}\label{n=2e}
\hat{\eta}[\bold{\Omega}](\bar{x})= \lim_{\delta \downarrow 0}\inf\left\{\norm{x_1^*+x_2^*} \mid x_i^* \in \widehat N_{\Omega_i}(\bar x,\delta), \norm{x_i^*}=\frac{1}{2}\; (i=1,2)\right\},
\end{equation}
where it is assumed that the infimum over the empty set equals $1$.
\end{proposition}
\begin{proof}
If, for some $\delta>0$, one of the cones $\widehat N_{\Omega_1}(\bar x,\delta)$ or $\widehat N_{\Omega_2}(\bar x,\delta)$ is trivial, then $\hat{\eta}[\bold{\Omega}](\bar{x})=1$ and the equality is satisfied automatically.
Take arbitrary nonzero $x_1^*\in \widehat N_{\Omega_1}(\bar x,\delta)$ and $x_2^*\in \widehat N_{\Omega_2}(\bar x,\delta)$ such that $\|x_1^*\|+\|x_2^*\|=1$.
Then
\begin{align*}
(\|x_1^*\|-\|x_2^*\|)^2 &=\|x_1^*\|^2+\|x_2^*\|^2-2\|x_1^*\|\|x_2^*\|,
\\
1&=\|x_1^*\|^2+\|x_2^*\|^2+2\|x_1^*\|\|x_2^*\|.
\end{align*}
Hence,
\begin{align*}
\|x_1^*\|^2+\|x_2^*\|^2 &=\frac{1+(\|x_1^*\|-\|x_2^*\|)^2}{2},
\\
\|x_1^*\|\|x_2^*\| &=\frac{1-(\|x_1^*\|-\|x_2^*\|)^2}{4}.
\end{align*}
Set
$$
z_1^*:=\frac{x_1^*}{2\|x_1^*\|}
\quad\mbox{and}\quad
z_2^*:=\frac{x_2^*}{2\|x_2^*\|}.
$$
Then $z_1^*\in \widehat N_{\Omega_1}(\bar x,\delta)$, $z_2^*\in \widehat N_{\Omega_2}(\bar x,\delta)$, $\norm{z_i^*}=\norm{z_2^*}=\frac{1}{2}$, and
$$
\norm{z_1^*+z_2^*}^2 =\frac{1}{2}+\frac{\langle x_1^*,x_2^*\rangle} {2\norm{x_1^*}\norm{x_2^*}}.
$$
Next we show that
$$
\norm{x_1^*+x_2^*}\ge \norm{z_1^*+z_2^*}.
$$
Indeed,
\begin{align*}
\norm{x_1^*+x_2^*}^2 &-\norm{z_1^*+z_2^*}^2 =\norm{x_1^*}^2+\norm{x_2^*}^2+2\langle x_1^*,x_2^*\rangle-\frac{1}{2}-\frac{\langle x_1^*,x_2^*\rangle} {2\norm{x_1^*}\norm{x_2^*}}
\\
&=\frac{1+(\|x_1^*\|-\|x_2^*\|)^2}{2}-\frac{1}{2}+2\langle x_1^*,x_2^*\rangle-\frac{\langle x_1^*,x_2^*\rangle} {2\norm{x_1^*}\norm{x_2^*}}
\\
&=\frac{(\|x_1^*\|-\|x_2^*\|)^2}{2} +\frac{4\norm{x_1^*}\norm{x_2^*}-1} {2\norm{x_1^*}\norm{x_2^*}}\langle x_1^*,x_2^*\rangle
\\
&=\frac{(\|x_1^*\|-\|x_2^*\|)^2}{2} -\frac{(\|x_1^*\|-\|x_2^*\|)^2} {2\norm{x_1^*}\norm{x_2^*}}\langle x_1^*,x_2^*\rangle
\\
&=\frac{(\|x_1^*\|-\|x_2^*\|)^2}{2} \left(1 -\frac{\langle x_1^*,x_2^*\rangle} {\norm{x_1^*}\norm{x_2^*}}\right)\ge0.
\end{align*}
This completes the proof.
\qed\end{proof}

The following example shows that the conclusion of Proposition~\ref{n=2} is not true in non-Hilbert spaces.
\begin{example} Consider $\R^2$ with the sum norm, $\|(x,y)\|=|x|+|y|$, and take $\Omega_1=\{(x_1,x_2)\mid x_2\le 0\}$, $\Omega_2=\{(x_1,x_2)\mid x_2\ge 2x_1\}$ and $\bar x =(0,0)\in \Omega_1\cap\Omega_2$. Then, for any $\delta>0$, we have
$$
\widehat N_{\Omega_1}(\bar x,\delta)= \{t(0,1)\mid t\in \R^+\},
$$
$$
\widehat N_{\Omega_2}(\bar x,\delta)= \{t(2,-1)\mid t\in \R^+\}.
$$
Thus,
$$
z_1^* \in \widehat N_{\Omega_1}(\bar x,\delta) \mbox{ with } \|z_1^*\|=\frac{1}{2} \Longrightarrow z_1^*=(0,1/2),
$$
$$
z_2^* \in \widehat N_{\Omega_2}(\bar x,\delta) \mbox{ with } \|z_2^*\|=\frac{1}{2} \Longrightarrow z_2^*=(1/3,-1/6),
$$
and the right-hand side of \eqref{n=2e} equals $\|z_1^*+z_2^*\| = \|(1/3,1/3)\| =2/3$.
At the same time, with $x_1^*=(0,1/4) \in \widehat N_{\Omega_1}(\bar x,\delta)$ and $x_2^*=(\frac{1}{2},-1/4) \in \widehat N_{\Omega_2}(\bar x,\delta)$ it holds $\|x_1^*\|+\|x_2^*\| = 1$ and $\|x_1^*+x_2^*\| = \frac{1}{2}$.
Hence,
$
\hat{\eta}[\bold{\Omega}](\bar{x})\le\|x_1^*+x_2^*\|< 2/3.
$
\qedtr\end{example}

The next proposition provides an equivalent representation of constant \eqref{ets}.

\begin{proposition}\label{vid}
The following representation holds true:
\begin{equation}\label{ba}
\hat c[\bold{\Omega}](\bar x)= \lim_{\delta\downarrow0}\sup\left\{-\langle x_1^*,x_2^*\rangle\mid x_i^*\in \widehat N_{\Omega_i}(\bar x,\delta), \norm{x_i^*}=1\; (i=1,2)\right\}.
\end{equation}
where it is assumed that the supremum over the empty set equals $-1$.
\end{proposition}

\begin{proof}
If, for some $\delta>0$, one of the cones $\widehat N_{\Omega_1}(\bar x,\delta)$ or $\widehat N_{\Omega_2}(\bar x,\delta)$ is trivial, then $\hat{\eta}[\bold{\Omega}](\bar{x})=1$, the \RHS\ of \eqref{ba} equals $-1$ and coincides with $\hat{c}[\bold{\Omega}](\bar{x})$ computed in accordance with definition \eqref{ets}.
Let both cones be nontrivial for any $\delta>0$.
Then, by \eqref{ets}, \eqref{n=2e}, and \eqref{ba},
\begin{align*}
\hat c[\bold{\Omega}](\bar x) &=\lim_{\delta\downarrow0}\sup \left\{1-2\norm{x_1^*+x_2^*}^2
\mid x_i^*\in \widehat N_{\Omega_i}(\bar x,\delta), \norm{x_i^*}=\frac{1}{2}\; (i=1,2)\right\}
\\
&=\lim_{\delta\downarrow0}\sup \left\{-\langle 2x_1^*,2x_2^*\rangle
\mid x_i^*\in \widehat N_{\Omega_i}(\bar x,\delta), \norm{x_i^*}=\frac{1}{2}\; (i=1,2)\right\}
\\
&=\hat{c}[\bold{\Omega}](\bar{x}).
\end{align*}
\qed\end{proof}

Another dual space constant can be used alongside \eqref{n=2e} and \eqref{ba} for characterizing the uniform regularity of a collection of two sets in a Hilbert space:
\begin{align}\label{constant2}
\hat\nu[\bold{\Omega}](\bar x):=& \lim_{\delta\downarrow0}\sup\left\{\norm{x_1^*-x_2^*}\mid x_i^*\in \widehat N_{\Omega_i}(\bar x,\delta), \norm{x_i^*}=\frac{1}{2}\; (i=1,2)\right\},
\end{align}
where it is assumed that the supremum over the empty set equals 0;
this corresponds to one of the cones $\widehat N_{\Omega_1}(\bar x,\delta)$ or $\widehat N_{\Omega_2}(\bar x,\delta)$ being trivial for some $\delta>0$ ($\bar x$ can be an interior point of either $\Omega_1$ or $\Omega_2$.)

\begin{remark}\label{unlike}
Unlike constants $\hat\eta[\bold{\Omega}](\bar x)$ and $\hat c[\bold{\Omega}](\bar x)$, the definition of constant $\hat\nu[\bold{\Omega}](\bar x)$ is specific for the case of two sets.
\end{remark}
\begin{remark}
Condition $\norm{x_i^*}=\frac{1}{2}$, $i=1,2$, in definition \eqref{constant2} cannot be replaced by $\norm{x_1^*}+\norm{x_2^*}=1$ (as in \eqref{constant1'}): it would always be equal to $1$.
\end{remark}

\begin{theorem}\label{relationship1'}
The following relations hold true:
\begin{enumerate}
\item
$(\hat\eta[\bold{\Omega}](\bar x))^2 +(\hat\nu[\bold{\Omega}](\bar x))^2=1$;
\item
$1+\hat c[\bold{\Omega}](\bar x) =2(\hat\nu[\bold{\Omega}](\bar x))^2$.
\end{enumerate}
\end{theorem}

\begin{proof}
If, for some $\delta>0$, one of the cones $\widehat N_{\Omega_1}(\bar x,\delta)$ or $\widehat N_{\Omega_2}(\bar x,\delta)$ is trivial, then $\hat{\eta}[\bold{\Omega}](\bar{x})=1$, $\hat{\nu}[\bold{\Omega}](\bar{x})=0$, $\hat{c}[\bold{\Omega}](\bar{x})=-1$, and equalities (i) and (ii) are satisfied automatically.
Let both cones be nontrivial for any $\delta>0$.
Fix an arbitrary $\varepsilon>0$.

(i) By definition \eqref{constant2}, there exists $\delta>0$ such that
$$
\|x_1^*-x_2^*\| \le
\hat\nu[\bold{\Omega}](\bar x)+\varepsilon
$$
for any $x_i^*\in \widehat N_{\Omega_i}(\bar x,\delta)$ with $\|x_i^*\|=\frac{1}{2}$ $(i=1,2)$.
At the same time, by \eqref{n=2e},
there are elements $x_i^*\in \widehat N_{\Omega_i}(\bar x,\delta)$ with $\|x_i^*\|=\frac{1}{2}$ $(i=1,2)$ such that
$$
\|x_1^{*}+x_2^{*}\| \le \hat\eta[\bold{\Omega}](\bar x) +\varepsilon.
$$
Hence,
$$
(\hat\eta[\bold{\Omega}](\bar x) +\varepsilon)^2+(\hat\nu[\bold{\Omega}](\bar x) +\varepsilon)^2\ge\|x_1^*-x_2^*\|^2+\|x_1^{*}+x_2^{*}\|^2 =1.
$$
Since $\varepsilon$ is arbitrary, we have $$\hat\eta[\bold{\Omega}](\bar x)^2+\hat\nu[\bold{\Omega}](\bar x)^2\ge 1.$$

Similarly, by \eqref{n=2e} and \eqref{constant2}, we find elements $x_i^*\in \widehat N_{\Omega_i}(\bar x,\delta)$ with $\|x_i^*\|=\frac{1}{2}$ $(i=1,2)$ such that
$$
\|x_1^{*}-x_2^{*}\| \ge \hat\nu[\bold{\Omega}](\bar x)-\varepsilon,
$$
$$
\|x_1^{*}+x_2^{*}\| \ge \hat\eta[\bold{\Omega}](\bar x)-\varepsilon.
$$
This yields
$$
(\hat\nu[\bold{\Omega}](\bar x)-\varepsilon)^2+(\hat\eta[\bold{\Omega}](\bar x)-\varepsilon)^2\le 1,
$$
and consequently,
$$\hat\eta[\bold{\Omega}](\bar x)^2+\hat\nu[\bold{\Omega}](\bar x)^2\le 1.$$

(ii) follows immediately from (i) and definition \eqref{ets}.
\qed\end{proof}

\begin{corollary}\label{C}
$\{\Omega_1,\Omega_2\}$ is uniformly regular at $\bar{x} \in \Omega_1\cap\Omega_2$ if and only if one of the following equivalent conditions holds true:
\begin{enumerate}
\item
$\hat\eta[\bold{\Omega}](\bar x)>0$;
\item
$\hat\nu[\bold{\Omega}](\bar x) < 1$;
\item
$\hat c[\bold{\Omega}](\bar x) < 1$.
\end{enumerate}
\end{corollary}

The next example shows that the equality in Theorem~\ref{relationship1'} (ii) remains true when $\hat c[\bold{\Omega}](\bar x)\le0$.

\begin{example}
In $\R^2$ with the Euclidean norm, we fix $\Omega_1=\{(x_1,x_2)\mid x_2\le 0\}$ and $\bar x=(0,0)$. Then, for any $\delta>0$, $\widehat N_{\Omega_1}(\bar x,\delta)=\{t(0,1)\mid t\ge 0\}$. We consider the following two cases of $\Omega_2$:

\emph{Case 1}. $\Omega_2=\{(x_1,x_2)\mid x_1\le 0\}$. For any $\delta>0$, $\widehat N_{\Omega_2}(\bar x,\delta)=\{t(1,0)\mid t\ge 0\}$. Then $\hat c[\bold{\Omega}](\bar x)=0$ and
$\hat\nu[\bold{\Omega}](\bar x)=\frac{\sqrt{2}}{2}$.

\emph{Case 2}. $\Omega_2=\{(x_1,x_2)\mid x_1+x_2\le0\}$. For any $\delta>0$, $\widehat N_{\Omega_2}(\bar x,\delta)=\{t(1,1)\mid t\ge 0\}$. Then $\hat c[\bold{\Omega}](\bar x)= -\frac{1}{\sqrt{2}}$ and
$\hat\nu[\bold{\Omega}](\bar x)= \frac{\sqrt{2-\sqrt{2}}}{2}$.

In both cases the equality in Theorem~\ref{relationship1'} (ii) holds true.
\qedtr\end{example}

\begin{remark}
In finite dimensions, constants \eqref{n=2e}--\eqref{ba} coincide with the corresponding ones defined in terms of limiting normals:
\begin{align}\notag
\bar\eta[\bold{\Omega}](\bar x):=& \min\left\{\norm{x_1^*+x_2^*}\mid x_i^*\in \overline{N}_{\Omega_i}(\bar x), \norm{x_i^*}=\frac{1}{2}\; (i=1,2)\right\},
\\\label{nu}
\bar\nu[\bold{\Omega}](\bar x):=& \max\left\{\norm{x_1^*-x_2^*}\mid x_i^*\in \overline{N}_{\Omega_i}(\bar x), \norm{x_i^*}=\frac{1}{2}\; (i=1,2)\right\},
\\\label{constant3}
\bar c[\bold{\Omega}](\bar x):=& \max\left\{-\langle x_1^*,x_2^*\rangle\mid x_i^*\in \overline{N}_{\Omega_i}(\bar x), \norm{x_i^*}=1\; (i=1,2)\right\}
\end{align}
(with the similar natural conventions about the minimum and maximum over the empty set.)
The relations amongst the above constants are consequences of those in Theorem~\ref{relationship1'}:
\begin{enumerate}
\item
$(\bar\eta[\bold{\Omega}](\bar x))^2 +(\bar\nu[\bold{\Omega}](\bar x))^2=1$;
\item
$1+\bar c[\bold{\Omega}](\bar x) =2(\bar\nu[\bold{\Omega}](\bar x))^2$;
\item
$1-\bar c[\bold{\Omega}](\bar x) =2(\bar\eta[\bold{\Omega}](\bar x))^2$.
\end{enumerate}
\end{remark}

\begin{remark}
Constant \eqref{constant3} is closely related with the one introduced in \cite{LewLukMal09}:
\begin{equation*}
\bar c:= \max\left\{-\langle x_1^*,x_2^*\rangle\mid x_i^*\in \overline{N}_{\Omega_i}(\bar x)\cap \mathbb{B}\; (i=1,2)\right\}.
\end{equation*}
Indeed,
$\bar c =(\bar c[\bold{\Omega}](\bar x))_+$,
where $(\alpha)_+:=\max\{\alpha,0\}$.
\end{remark}

Given a collection of $m$ sets $\bold{\Omega}=\{\Omega_1,\Omega_2,\ldots,\Omega_m\}$ in a finite dimensional Hilbert space $X$ and a point $\bar{x}\in \cap_{i=1}^m\Omega_i$, one can consider the Hilbert space $X^m$ with the norm
$$
\|(x_1,x_2,\ldots,x_n)\| =\left(\sum_{i=1}^m\|x_i\|^2\right)^{\frac{1}{2}}
$$
and compute constants \eqref{n=2e}, \eqref{ba}, and \eqref{constant2} corresponding to the collection $\bold{\Omega'}:=\{\Omega,L\}$ and the point $\bar{z}:=A\bar{x} =(\bar{x},\bar{x},\ldots,\bar{x})\in\Omega\cap L$, where $\Omega$ and $L$ are defined by \eqref{product}.

\begin{proposition}\label{rewrite}
The following representations hold true:
\begin{multline}\label{eta'}
\hat{\eta}[\bold{\Omega'}](\bar{z})= \lim_{\delta \downarrow 0}\inf \Biggl{\{}\left(\frac{1}{2}-\frac{1}{2}\left(1 -\frac{1}{m}\norm{x_1^*+\ldots+x_m^*}^2\right)^{\frac{1}{2}}\right)^{\frac{1}{2}}\mid
\\
x_i^* \in \widehat N_{\Omega_i}(\bar x,\delta)\; (1\le i\le{m}), \sum_{i=1}^m\|x_i^*\|^2=1\Biggr{\}},
\end{multline}
\begin{multline}\label{nu'}
\hat{\nu}[\bold{\Omega'}](\bar{z})= \lim_{\delta \downarrow 0}\sup \Biggl{\{}\left(\frac{1}{2}+\frac{1}{2}\left(1-\frac{1}{m}\norm{x_1^*+\ldots+x_m^*}^2\right)^{\frac{1}{2}}\right)^{\frac{1}{2}}\mid
\\
x_i^* \in \widehat N_{\Omega_i}(\bar x,\delta)\; (1\le i\le{m}), \sum_{i=1}^m\|x_i^*\|^2=1\Biggr{\}},
\end{multline}
\begin{multline}\label{c'}
\hat{c}[\bold{\Omega'}](\bar{z})= \lim_{\delta \downarrow 0}\sup \Biggl{\{}\left(1-\frac{1}{m} \norm{x_1^*+\ldots+x_m^*}^2\right)^{\frac{1}{2}}\mid \\
x_i^* \in \widehat N_{\Omega_i}(\bar x,\delta)\; (1\le i\le{m}),\;
\sum_{i=1}^m\|x_i^*\|^2=1\Biggr{\}}.
\end{multline}
\end{proposition}
\begin{proof}
If $z_1=(x_1,x_2,\ldots,x_n)$, $z_2=(u_1,u_2,\ldots,u_n)\in X^m$, then
$$
\|z_1+z_2\|^2 =\sum_{i=1}^m\|x_i\|^2+\sum_{i=1}^m\|u_i\|^2 +2\sum_{i=1}^m\langle x_i,u_i\rangle.
$$

By the structure of $\bold{\Omega'}$ and \eqref{n=2e}, we have
\begin{align}\notag
\hat{\eta}[\bold{\Omega'}](\bar{z})
&= \lim_{\delta \downarrow 0}\inf \Biggl\{\left(\frac{1}{2}+2\sum_{i=1}^m\langle x_i^*,u_i\rangle\right)^{\frac{1}{2}}\mid
\sum_{i=1}^m\|x_i^*\|^2 =\sum_{i=1}^m\|u_i\|^2=\frac{1}{4},
\\\notag
&\qquad\qquad\qquad\qquad x_i^* \in \widehat N_{\Omega_i}(\bar x,\delta),\sum_{i=1}^mu_i=0\; (1\le i\le{m})\Biggr\}
\\\notag
&= \lim_{\delta \downarrow 0}\inf \Biggl\{\left(\frac{1}{2}+\frac{1}{2}\sum_{i=1}^m\langle x_i^*,u_i\rangle\right)^{\frac{1}{2}}\mid
\sum_{i=1}^m\|x_i^*\|^2=\sum_{i=1}^m\|u_i\|^2=1,
\\\label{pre1}
&\qquad\qquad\qquad\qquad x_i^* \in \widehat N_{\Omega_i}(\bar x,\delta),\sum_{i=1}^mu_i=0\; (1\le i\le{m})\Biggr\}.
\end{align}

Fix any $x_i^* \in \widehat N_{\Omega_i}(\bar x,\delta)$ $(1\le i\le{m})$ with $\sum_{i=1}^m\|x_i^*\|^2=1$ and denote \begin{gather}\label{rew0}
x_0^*:=\frac{1}{m}\sum_{i=1}^m x_i^*.
\end{gather}
Consider the following minimization problem in $X^m$ which is an important component of \eqref{pre1}:
\begin{align*}
&\mbox{minimize }\quad f(u):=\sum_{i=1}^m\langle x_i^*,u_i\rangle
\\
&\mbox{subject to}\quad\sum_{i=1}^mu_i=0\mbox{ and } \sum_{i=1}^m\|u_i\|^2=1.
\end{align*}
Since $f$ is continuous and the constraint set is compact, the above problem has a solution $u^\circ=(u_1^\circ,u_2^\circ,\ldots,u_m^\circ)$.
In accordance with the Lagrange multiplier rule, there exist multiplies $\lambda_0$, $\lambda_1\in\R$ and $u^*\in X$, not all zero, such that
\begin{gather}\label{rew1}
\lambda_0x_i^*+2\lambda_1u_i^\circ+u^*=0 \quad(1\le i\le m).
\end{gather}
Adding the equalities together and taking into account that $\sum_{i=1}^mu_i^\circ=0$, we obtain
\begin{gather}\label{rew2}
\lambda_0\sum_{i=1}^mx_i^*+mu^*=0.
\end{gather}
If $\lambda_0=0$, then $u^*=0$ and consequently $\lambda_1\ne0$ and, by \eqref{rew1}, $u_i^\circ=0$ for all $i\in\{1,2,\ldots,m\}$, which is impossible thanks to $\sum_{i=1}^m\|u_i^\circ\|^2=1$.
Hence, $\lambda_0\ne0$ and we can take $\lambda_0=1$.
It follows from \eqref{rew1}, \eqref{rew2}, and \eqref{rew0} that
\begin{gather}\label{rew3}
x_i^*+2\lambda_1u_i^\circ=x_0^* \quad(1\le i\le m),
\end{gather}
and consequently
\begin{align}\notag
4\lambda_1^2 =\sum_{i=1}^m \|x_0^*-x_i^*\|^2
&=m\|x_0^*\|^2+\sum_{i=1}^m \|x_i^*\|^2-2\left\langle\sum_{i=1}^m x_i^*,x_0^*\right\rangle
\\\label{rew4}
&=\sum_{i=1}^m \|x_i^*\|^2-m\|x_0^*\|^2.
\end{align}
At the same time,
\begin{align*}
2\lambda_1f(u^\circ)=\sum_{i=1}^m\langle x_i^*,2\lambda_1u_i^\circ\rangle &=\sum_{i=1}^m\langle x_i^*,x_0^*-x_i^*\rangle
\\
&=\left(\left\langle\sum_{i=1}^m x_i^*,x_0^*\right\rangle-\sum_{i=1}^m\|x_i^*\|^2\right) \\
&=\left(m\|x_0^*\|^2-\sum_{i=1}^m\|x_i^*\|^2\right) =-4\lambda_1^2.
\end{align*}
This yields either $f(u^\circ)=-2\lambda_1$ or $\lambda_1=0$.
In the last case, by \eqref{rew3}, $x_i^*=x_0^*$ for all $i\in\{1,2,\ldots,m\}$, and consequently
\begin{gather*}
f(u^\circ)=\sum_{i=1}^m \langle x_0^*,u_i^\circ\rangle =\left\langle x_0^*,\sum_{i=1}^mu_i^\circ\right\rangle=0.
\end{gather*}
Hence, in both cases, $f(u^\circ)=-2\lambda_1$.
Since $u^\circ$ is a point of minimum, $\lambda_1$ must be nonnegative, and consequently, by \eqref{rew4},
\begin{gather*}
f(u^\circ)=-\left(\sum_{i=1}^m \|x_i^*\|^2-m\|x_0^*\|^2\right)^{\frac{1}{2}} =-\left(1-m\|x_0^*\|^2\right)^{\frac{1}{2}}.
\end{gather*}
Combining this with \eqref{pre1}, we get \eqref{eta'}.

\eqref{nu'} and \eqref{c'} follow from \eqref{eta'} thanks to Theorem~\ref{relationship1'}.
\qed\end{proof}

\begin{corollary}\label{valuerange}
The following estimates hold true:
\begin{gather*}
0\le\hat{\eta}[\bold{\Omega'}](\bar{z})\le \left(\frac{1}{2}- \frac{1}{2}\sqrt{1-\frac{1}{m}}\right)^{\frac{1}{2}};\\
\left(\frac{1}{2}+ \frac{1}{2}\sqrt{1-\frac{1}{m}}\right)^{\frac{1}{2}}
\le
\hat{\nu}[\bold{\Omega'}](\bar{z})\le1;\\ \sqrt{1-\frac{1}{m}}\le
\hat{c}[\bold{\Omega'}](\bar{z})\le1. \end{gather*}
\end{corollary}
\begin{proof}
The estimates follow from Proposition \ref{rewrite} due to the fact that

\qquad $\min\{\|x_1+x_2+\ldots+x_m\|\mid \|x_1\|^2+\|x_2\|^2+\ldots+\|x_m\|^2=1\}\le 1$.
\qed\end{proof}

Dual space constants \eqref{eta'}, \eqref{nu'}, and \eqref{c'} can be used to characterize the uniform regularity of collections of $m$ sets.


The next corollary follows from Proposition \ref{mTO2} and Corollary \ref{C}.

\begin{corollary}\label{C2} $\bold{\Omega}$ is uniformly regular at $\bar x\in \cap_{i=1}^m\Omega_i$ if and only if one of the following equivalent conditions holds true:
\begin{enumerate}
\item
$\hat\eta[\bold{\Omega'}](\bar z)>0$;
\item
$\hat\nu[\bold{\Omega'}](\bar z) < 1$;
\item
$\hat c[\bold{\Omega'}](\bar z) < 1$.
\end{enumerate}
\end{corollary}

Observe that, when $m=2$, constants \eqref{eta'}, \eqref{nu'}, and \eqref{c'} do not coincide with the corresponding constants \eqref{n=2e}, \eqref{constant2}, and \eqref{ba} .

\begin{corollary}\label{compare}
When $m=2$, the following relations hold true:
\begin{align}\notag
\hat{\eta}[\bold{\Omega'}](\bar{z}) &=\lim_{\delta \downarrow 0}\inf \Biggl\{\left(\frac{1-\|x_1^*-x_2^*\|}{2}\right)^{\frac{1}{2}} \mid&& x_i^* \in \widehat N_{\Omega_i}(\bar x,\delta)\; (i=1,2),
\\\notag
&&&\|x_1^*\|^2+\|x_2^*\|^2=\frac{1}{2}\Biggr\},
\\\notag
\hat{\nu}[\bold{\Omega'}](\bar{z}) &=\lim_{\delta \downarrow 0}\sup \Biggl\{\left(\frac{1+\|x_1^*-x_2^*\|}{2}\right)^{\frac{1}{2}} \mid&& x_i^* \in \widehat N_{\Omega_i}(\bar x,\delta)\; (i=1,2),
\\\notag
&&&\|x_1^*\|^2+\|x_2^*\|^2=\frac{1}{2}\Biggr{\}},
\\\notag
\hat{c}[\bold{\Omega'}](\bar{z})&= \lim_{\delta \downarrow 0}\sup \biggl{\{}\|x_1^*-x_2^*\|\quad\mid&& x_i^* \in \widehat N_{\Omega_i}(\bar x,\delta)\; (i=1,2),
\\\label{cons}
&&&\|x_1^*\|^2+\|x_2^*\|^2=\frac{1}{2}\biggr{\}}.
\end{align}
\end{corollary}
\begin{proof}
From Proposition \ref{rewrite}, we have
\begin{align*}
\hat{c}[\bold{\Omega'}](\bar{z})= \lim_{\delta \downarrow 0}\sup \Biggl{\{}\left(1-\frac{1}{2} \norm{x_1^*+x_2^*}^2\right)^{1/2}\mid\ & x_i^*\in\widehat N_{\Omega_i}(\bar x,\delta)\; (i=1,2),
\\
&\|x_1^*\|^2+\|x_2^*\|^2=1\Biggr{\}}.
\end{align*}
In the above formula,
\begin{align*}
1-\frac{1}{2}\norm{x_1^*+x_2^*}^2 &=\frac{1}{2}(2-\norm{x_1^*+x_2^*}^2)
\\
&=\frac{1}{2}\left(2(\|x_1^*\|^2+\|x_2^*\|^2) -(\|x_1^*\|^2+\|x_2^*\|^2+2\langle x_1^*,x_2^*\rangle)\right)
\\
&=\frac{1}{2}(\|x_1^*\|^2+\|x_2^*\|^2 -2\langle x_1^*,x_2^*\rangle) \\
&=\frac{1}{2}\norm{x_1^*-x_2^*}^2 =\left\|\frac{x_1^*}{\sqrt{2}} -\frac{x_2^*}{\sqrt{2}}\right\|^2
\end{align*}
and
\begin{align*}
\left\|\frac{x_1^*}{\sqrt{2}}\right\|^2 +\left\|\frac{x_2^*}{\sqrt{2}}\right\|^2 &=\frac{1}{2}.
\end{align*}
This proves \eqref{cons}, which also implies the other relations.
\qed\end{proof}

The next relation between $\hat{c}[\bold{\Omega'}](\bar{z})$ and $\hat{\nu}[\bold{\Omega}](\bar{x})$ can be of interest.

\begin{proposition}
When $m=2$, it holds:
\begin{equation}\label{relation}
\hat{c}[\bold{\Omega'}](\bar{z})\ge \hat{\nu}[\bold{\Omega}](\bar{x}).
\end{equation}
Furthermore, \eqref{relation} holds as an equality whenever $\hat{\nu}[\bold{\Omega}](\bar{x}) > 1/\sqrt{2}$.
\end{proposition}
\begin{proof}
In view of \eqref{cons} and \eqref{constant2}, inequality \eqref{relation} is always true.

We prove the second assertion.
Suppose
$\hat{\nu}[\bold{\Omega}](\bar{x}) > 1/\sqrt{2}$.
By \eqref{constant2}, for any $\delta>0$, one can find $x_i^* \in \widehat N_{\Omega_i}(\bar x,\delta)$ with $\norm{x_i^*}=\frac{1}{2}\;(i=1,2)$ such that
$\|x_1^* - x_2^*\|>1/\sqrt{2}$.

Observe that, for any $x_1^*$ and $x_2^*$ with $\norm{x_1^*}^2+\norm{x_2^*}^2=\frac{1}{2}$, it holds
$$
\norm{x_1^*-x_2^*}^2=\frac{1}{2}-2\langle x_1^*,x_2^*\rangle.
$$
Hence, maximizing $\norm{x_1^*-x_2^*}$ is equivalent to minimizing $\langle x_1^*,x_2^*\rangle$, and condition $\|x_1^* - x_2^*\|>1/\sqrt{2}$ is equivalent to $\langle x_1^*,x_2^*\rangle<0$.
Under the assumptions made,
\begin{align*}
&\sup\Big\{\norm{x_1^*-x_2^*}\mid x_i^*\in \widehat N_{\Omega_i}(\bar x,\delta)\, (i=1,2),\,
\norm{x_1^*}^2+\norm{x_2^*}^2=\frac{1}{2}\Big\}
\\
=&\sup\Big\{\norm{x_1^*-x_2^*}\mid x_i^*\in \widehat N_{\Omega_i}(\bar x,\delta)\, (i=1,2),\,
\norm{x_1^*}^2+\norm{x_2^*}^2=\frac{1}{2},\,\langle x_1^*,x_2^*\rangle<0\Big\}
\\
=&\sup\Big\{\norm{x_1^*-x_2^*}\mid x_i^*\in \widehat N_{\Omega_i}(\bar x,\delta),\,\norm{x_i^*}=\frac{1}{2}\, (i=1,2),\,
\langle x_1^*,x_2^*\rangle<0\Big\},
\end{align*}
and it follows from \eqref{cons} that
$\hat{c}[\bold{\Omega'}](\bar{z})= \hat{\nu}[\bold{\Omega}](\bar{x})$.
\qed\end{proof}

\section{Applications in projection algorithms}\label{S3}

Inspired by \cite{LewLukMal09}, we are making an attempt to extend convergence results of the alternating projections for solving feasibility problems to those of the cyclic projection algorithms in Hilbert spaces.
Recall that a feasibility problem consists in finding common points of a collection of sets with nonempty intersection.
This model incorporates many important optimization problems.

We first recall some basic facts about projections.
Given a nonempty set $\Omega$ in a normed linear space $X$, the distance function and projection mapping are defined, for $x\in X$, respectively, as follows:
$$
d(x,\Omega):= \inf_{\omega\in \Omega}\norm{x-\omega},
$$
$$
P_{\Omega}(x):= \left\{\omega\in\Omega \mid \|x-\omega\|=d(x,\Omega)\right\}.
$$

\begin{lemma}[\cite{ClaLedSteWol98}]\label{lem1m}
$\omega\in P_{\Omega}(x) \Longrightarrow x-\omega \in {N}_{\Omega}(\omega)$.
\end{lemma}

From now on, we are considering a finite collection of closed sets $\bold{\Omega}=\{\Omega_1,\Omega_2,\ldots,\Omega_m\}$ ($m>1$) and assuming the existence of a point $\bar{x}\in \cap_{i=1}^m\Omega_i$.

\begin{definition}\label{en}
A sequence $(x_k)$ is generated by
\begin{enumerate}
\item
the \emph{averaged projections} for $\bold{\Omega}$ if
    \begin{equation}\label{averagedP}
    x_{k+1} \in \dfrac{1}{m} \sum_{i=1}^m P_{\Omega_i}(x_k),\quad k=0,1,\ldots;
    \end{equation}
\item
the \emph{cyclic projections} for $\bold{\Omega}$ if
    \begin{equation}\label{cyclicP}
    x_{k+1} \in P_{\Omega_{k+1}}(x_k),\quad k=0,1,\ldots,
    \end{equation}
with the convention $\Omega_{i+nm}=\Omega_i$ for all $i=1,\ldots,m$ and $n\in\N$.
\end{enumerate}
\end{definition}

Note that the existence of the sequences in Definition~\ref{en} cannot be guaranteed in general, unless the space is finite dimensional.

From now on, we are assuming that $X$ is a Hilbert space.
The next regularity property is needed in our analysis.

\begin{definition}[\cite{LewLukMal09}, Definition 4.3] \label{d3}
A closed set $\Omega$ is super-regular at $\bar x\in \Omega$ if, for any $\gamma>0$, any two points $x,z$ sufficiently close to $\bar x$ with $z\in \Omega$, and any point $y\in P_{\Omega}(x)$, it holds $\langle z-y,x-y\rangle\le \gamma\|z-y\|\cdot\|x-y\|$.
\end{definition}
\begin{lemma}[\cite{LewLukMal09}, Proposition 4.4]\label{l2}
A closed set $\Omega$ is super-regular at $\bar x\in \Omega$ if and only if for any $\gamma>0$, there is $\delta>0$ such that
$$
\langle u,z-x\rangle\le \gamma\|u\|\cdot\|z-x\|,\quad \forall z,x\in \Omega\cap B_{\delta}(\bar x), u\in N_{\Omega}(x).
$$
\end{lemma}

\begin{remark} 
Similar to the well known \emph{prox-regularity} property (the projection mapping associated with the set being single-valued around the reference point; cf. \cite{ClaSteWol95,BerThi04,Sha94,PolRocThi00}), the super-regu\-larity one in Definition~\ref{d3} is a way of describing sets being locally ``almost'' convex.
It is weaker than the prox-regularity while stronger than the Clarke regularity and fits well the convergence analysis of projections algorithms.
For a detailed discussion and characterizations of this property we refer the reader to \cite{LewLukMal09}.
\end{remark}

\begin{theorem}\label{newconvergence}
Suppose $\bold{\Omega}$ is uniformly regular at $\bar x$ with
\begin{equation}\label{condition}
\hat c[\bold{\Omega}](\bar x)<\frac{1}{m-1}
\end{equation}
and $\Omega_1$ is super-regular at $\bar x$.
Then, for any $c\in((m-1)\hat c[\bold{\Omega}](\bar x),1)$,
a sequence $(x_k)$ generated by cyclic projections for $\bold{\Omega}$ linearly converges to some point in $\cap_{i=1}^m\Omega_i$ with rate $\sqrt[m]{c}$, provided that for each $k=0,1,\ldots,$
\begin{equation}\label{nonexpansive}
\|x_{km+i+1}-x_{km+i}\|\leq \|x_{km+2}-x_{km+1}\|\quad (i=2,\ldots,m)
\end{equation}
and $x_0$ is sufficiently close to $\bar x$.
\end{theorem}

\begin{proof}
Let $c\in((m-1)\hat c[\bold{\Omega}](\bar x),1)$.
Choose $c'>\hat c[\bold{\Omega}](\bar x)$ and $\gamma>0$ such that $(m-1)c'+m\gamma<c$ and $\delta> 0$ such that the conclusions of Lemmas \ref{l1} and \ref{l2} (with $\Omega=\Omega_1$) are satisfied.

Let $x_0\in X$ be such that
$$\|x_0-\bar x\| <\frac{\delta(1-c)}{2(m+1)}.$$
Then
\begin{equation}\label{estimate2}
\alpha:=\|x_1-\bar x\|\le\|x_1-x_0\|+\|x_0-\bar x\| \le2\|x_0-\bar x\|<\frac{\delta(1-c)}{m+1}
\end{equation}
and, by (\ref{nonexpansive}),
\begin{align}\notag
\|x_i-\bar x\|&\le \|x_i-x_{i-1}\|+\ldots+\|x_1-\bar x\|
\\\notag
&\le (i-1)\|x_2-x_{1}\|+\|x_1-\bar x\| \le i\|x_1-\bar x\|
\\\label{ans}
&=i\alpha\le(m+1)\alpha\quad (i=2,\ldots,m+1).
\end{align}
We are going to prove by induction that, for all $k=0,1,\ldots$,
\begin{equation}\label{estimate1}
\|x_{km+i}-\bar x\|\le (m+1)\alpha\frac{1-c^{k+1}}{1-c}\quad (i=2,\ldots,m+1).
\end{equation}
When $k=0$, the required inequalities have been established in \eqref{ans}.
Supposing that the inequalities are true for all $k=0,\ldots,l$ where $l\ge0$,
we show that they hold true for $k=l+1$.

We first prove that
\begin{align}\label{mma}
\|x_{(k+1)m+1}-x_{(k+1)m}\| \leq c\|x_{km+2}-x_{km+1}\|\quad (k=0,\ldots,l).
\end{align}
Indeed, if $x_{(k+1)m+1}=x_{(k+1)m}$, the inequality is trivially satisfied.
If $x_{km+2}=x_{km+1}$, then, by condition (\ref{nonexpansive}), $x_{(k+1)m+1}=x_{(k+1)m}$, and the inequality is satisfied too.
Otherwise,
by \eqref{estimate2} and \eqref{estimate1}, $\|x_{km+i}-\bar x\|<\delta$ $(i=2,\ldots,m+1)$ and
we have by Lemmas \ref{l1} and \ref{lem1m}, condition (\ref{nonexpansive}) and definition of projections: \begin{align*}
\langle x_{(k+1)m}&-x_{(k+1)m+1},x_{km+i+1}-x_{km+i}\rangle
\\
&< c' \|x_{(k+1)m}-x_{(k+1)m+1}\|\cdot\|x_{km+i+1}-x_{km+i}\|
\\
&\le c' \|x_{(k+1)m}-x_{(k+1)m+1}\|\cdot\|x_{km+2}-x_{km+1}\|\quad (i=1,\ldots,m-1).
\end{align*}
Adding the above inequalities, we obtain
\begin{multline}\label{lit}
\langle x_{(k+1)m}-x_{(k+1)m+1},x_{(k+1)m}-x_{km+1}\rangle
\\
< (m-1)c'\|x_{(k+1)m}-x_{(k+1)m+1}\|\cdot\|x_{km+2}-x_{km+1}\|.
\end{multline}
At the same time, by Lemma~\ref{l2}, the triangle inequality and condition (\ref{nonexpansive}),
\begin{align*}
\langle x_{(k+1)m}-x_{(k+1)m+1},&x_{km+1}-x_{(k+1)m+1}\rangle
\\
&\le \gamma\|x_{(k+1)m}-x_{(k+1)m+1}\|\cdot \|x_{(k+1)m+1}-x_{km+1}\|,
\\
\|x_{(k+1)m+1}-x_{km+1}\| &\le \sum_{i=1}^{m}\|x_{km+i+1}-x_{km+i}\|
\le m\|x_{km+2}-x_{km+1}\|,
\end{align*}
and consequently,
\begin{multline}\label{qre}
\langle x_{(k+1)m}-x_{(k+1)m+1},x_{km+1}-x_{(k+1)m+1}\rangle
\\
\le m\gamma\|x_{(k+1)m}-x_{(k+1)m+1}\|\cdot \|x_{km+2}-x_{km+1}\|.
\end{multline}
Adding \eqref{lit} and \eqref{qre}, we get
$$
\|x_{(k+1)m}-x_{(k+1)m+1}\|^2 < c\|x_{(k+1)m}-x_{(k+1)m+1}\|\cdot \|x_{km+2}-x_{km+1}\|,
$$
or equivalently
$$
\|x_{(k+1)m+1}-x_{(k+1)m}\|<c\|x_{km+2}-x_{km+1}\|.
$$
This proves \eqref{mma}.

Now with $k=l+1$ and taking into account \eqref{mma}, we have for $i=2,\ldots,m+1$:
\begin{align*}
\|x_{(l+1)m+i}-\bar x\|&\le \|x_{(l+1)m+i}-x_{(l+1)m+i-1}\|+\ldots+\|x_{(l+1)m}-\bar x\|
\\
&\le i\|x_{(l+1)m+1}-x_{(l+1)m}\|+\|x_{(l+1)m}-\bar x\|
\\
&\le ic^{l+1}\|x_{2}-x_{1}\|+\|x_{lm+m}-\bar x\|
\\
&\le (m+1)\alpha c^{l+1}+(m+1)\alpha\frac{1-c^{l+1}}{1-c} =(m+1)\alpha\frac{1-c^{l+2}}{1-c}.
\end{align*}

Finally we prove that $(x_n)$ converges to some point $\tilde x$ in $\cap_{i=1}^m\Omega_i$ with rate $\sqrt[m]{c}$.
Take any $k,r\in\N$ with $k>r$ and choose $n\in\N$ and $i\in \{0,1,\ldots,m-1\}$ such that $r=nm+i$.
We have
\begin{multline}\label{tht}
\|x_k-x_{r}\| \le \sum_{j=r}^{k-1}\|x_{j+1}-x_j\| \le \sum_{j=nm}^{\infty}\|x_{j+1}-x_j\|
\\
\le \sum_{j=n}^{\infty}\sum_{i=0}^{m-1} \|x_{mj+i+1}-x_{mj+i}\|
\le m\sum_{j=n}^{\infty}\|x_{mj+1}-x_{mj}\|
\\
\le m\|x_2-x_1\|\sum_{j=n}^{\infty}c^j
\le \frac{m\alpha c^n}{1-c}.
\end{multline}
Hence, $\|x_k-x_{r}\|\to0$ as $k,r\to\infty$, and consequently
$(x_n)$ is a Cauchy sequence and, therefore, converges to some point $\tilde x\in X$.
It follows from \eqref{tht} that
$$
\|\tilde x-x_r\|\le \frac{m\alpha c^n}{1-c} =\frac{m\alpha}{(1-c)c^\frac{i}{m}}c^\frac{r}{m}
\le \frac{m\alpha}{(1-c)c}(\sqrt[m]{c})^r.
$$
Finally, we check that $\tilde x\in\cap_{i=1}^m\Omega_i$.
Indeed, for any $i\in\{1,2,\ldots,m\}$, $x_{nm+i}\in\Omega_i$.
At the same time, $x_{nm+i}\to\tilde x$ as $n\to\infty$, and consequently, by the closedness of $\Omega_i$, $\tilde x\in\Omega_i$.
\qed\end{proof}

\begin{remark}
When $m=2$, conditions \eqref{nonexpansive} and \eqref{condition} are satisfied automatically.
In the general case of $m$ sets, condition \eqref{nonexpansive} can be ensured by, e.g., the next monotonicity condition:
$$
\|x_{k+2}-x_{k+1}\|\leq \|x_{k+1}-x_k\| \quad (k=1,2,\ldots,).
$$
\end{remark}

The convergence result of the alternating projection method, i.e., the cyclic projection method (\ref{cyclicP}) when $m=2$, established in \cite[Theorem 5.16]{LewLukMal09} is a consequence of Theorem~\ref{newconvergence}.

\begin{corollary}\label{MR} Suppose that $\bold{\Omega}$ is uniformly regular at $\bar x\in \Omega_1\cap \Omega_2$ and $\Omega_1$ is super-regular at this point. Then, any sequence generated by the alternating projections for $\bold{\Omega}$ linearly converges to some point in the intersection provided that $x_0$ is sufficiently close to $\bar x$.
\end{corollary}

Now, we derive from Corollary \ref{MR} another convergence result of the averaged projection algorithm for a collection of $m$ sets. Given a collection of sets $\bold{\Omega}=\{\Omega_1,\Omega_2,\ldots,\Omega_m\}$ in $X$, we consider the collection $\bold{\Omega'}:=\{\Omega,L\}$ of two sets in $X^m$ given by \eqref{product}. For $x\in X$, denote $Ax:= (x,x,\ldots,x) \in L$.

\begin{lemma}\label{projectionrelation}
\begin{enumerate}
\item
For any $x \in X$,
$$
P_{\Omega}(Ax) = \left(P_{\Omega_1}(x),P_{\Omega_2}(x),\ldots,P_{\Omega_m}(x)\right).
$$
\item
For any $(x_1,x_2,\ldots,x_m)\in X^m$,
$$
P_{L}(x_1,x_2,\ldots,x_m) = A\left(\frac{x_1+x_2+\ldots+x_m}{m}\right).
$$
\end{enumerate}
\end{lemma}

\begin{proof}
The first assertion is straightforward (cf. \cite[Exercise 1.8]{ClaLedSteWol98}). To prove the second one, we consider the real-valued function $f:X\to \R$ defined by
$$
f(x):= \sum_{i=1}^m\norm{x-x_i}^2.
$$
It is obvious that $Ax\in P_{L}(x_1,x_2,\ldots,x_m)$ if and only if $x$ is a minimizer of $f$. The conclusion follows from the first order optimality condition.
\qed\end{proof}

\begin{corollary}[\cite{LewLukMal09}, Theorem~7.3]
Suppose that $\bold{\Omega}$ is uniformly regular at $\bar{x}\in\cap_{i=1}^m \Omega_i$. Then any sequence $(y_k)$ generated by algorithm (\ref{averagedP}) linearly converges to some point in $\cap_{i=1}^m \Omega_i$ provided that the initial point $y_0$ is sufficiently close to $\bar x$.
\end{corollary}
\begin{proof}
Let $(z_n)$ be the sequence generated by the alternating projections for the two sets $\Omega$ and $L$ with the initial point $z_1:=Ay_1$. By Lemma \ref{projectionrelation},
$z_{2k}=Ay_k$, $k=1,2,\ldots$, for some sequence $(y_n)\subset X$.
At the same time, $\{\Omega,L\}$ is uniformly regular at $A\bar{x}$ by Proposition \ref{mTO2}. Therefore, when $y_0$ is sufficiently close to $\bar x$, Corollary \ref{MR} implies that the sequence $(z_n)$ linearly converges to some point $A\tilde{x}\in \Omega_1\cap \Omega_2$. It follows that the subsequence $(z_{2k}=Ay_k)$ also linearly converges to $A\tilde{x}$. Hence, $(y_k)$ linearly converges to $\tilde{x}\in \cap_{i=1}^m \Omega_i$.
\qed\end{proof}

\begin{acknowledgements}
The authors wish to thank the referee for helpful remarks and suggestions. 

The research was supported by the Australian Research Council, project DP110102011.
\end{acknowledgements}
\bibliographystyle{acm}
\bibliography{BUCH-Kr,Kruger,KR-tmp}
\end{document}